\documentclass{amsart}

\usepackage[breakable,skins]{tcolorbox}
\newtcolorbox{tbox}[1][]{%
    breakable,
    enhanced,
    colframe=blue,
    coltitle=white,
    #1
}

\usepackage[all]{xy}
\usepackage{hyperref}
\usepackage{amssymb}
\usepackage{mathdots}

\hypersetup{
    colorlinks,    citecolor=blue,    filecolor=blue,    linkcolor=blue,    urlcolor=blue
}

\usepackage{mathrsfs}
\usepackage{todonotes}
\usepackage{longtable}

\newtheorem{introthm}{Theorem}

\newtheorem{theorem}{Theorem}[section]
\newtheorem{lemma}[theorem]{Lemma}
\newtheorem{proposition}[theorem]{Proposition}
\newtheorem{corollary}[theorem]{Corollary}
\newtheorem{conjecture}[theorem]{Conjecture}
\theoremstyle{definition}

\newtheorem{definition}[theorem]{Definition}
\newtheorem{example}[theorem]{Example}

\newtheorem{remark}[theorem]{Remark}
\theoremstyle{remark}

\title{Equivalent conjectures
on blowing-ups of $\mathbb P^2$}

\author[A.~Laface]{Antonio Laface}
\address{
Departamento de Matem\'atica,
Universidad de Concepci\'on,
Casilla 160-C,
Concepci\'on, Chile}
\email{alaface@udec.cl}

\author[L.~Ugaglia]{Luca Ugaglia}
\address{
Dipartimento di Matematica e Informatica,
Universit\`a degli studi di Palermo,
Via Archirafi 34,
90123 Palermo, Italy}
\email{luca.ugaglia@unipa.it}

\author[M.~Vilches]{Macarena Vilches}
\address{
Departamento de Matem\'atica,
Universidad de Concepci\'on,
Casilla 160-C,
Concepci\'on, Chile}
\email{mvilches2016@udec.cl}

\thanks{The authors have been 
partially supported by Proyecto
FONDECYT Regular 
n.~1230287 and by 
``Piano straordinario per il miglioramento della qualit\`a della ricerca e dei risultati della VQR 2020-2024 - Misura A''
of the University of Palermo.
The second author is member of 
INdAM - GNSAGA}

\numberwithin{equation}{section}

\begin{document}

\begin{abstract}
We provide a characterization of asymptotical 
speciality of a nef and big divisor $D$
on an algebraic surface in terms of the
arithmetic genus of curves in $D^{\perp}$.
As a consequence we prove that the SHGH 
conjecture for linear systems on
the blowing-up of the projective plane
at points in very general position
is equivalent to the fact that each nef
class is non-special.
Finally we prove that if $r < 2^n$ then any nef divisor of the blowing-up of the $n$-dimensional
projective space 
at $r$ points in very general position is asymptotically non-special.
\end{abstract}

\maketitle

\section*{Introduction}
Let $\pi\colon X_r^n\to\mathbb P^n$
be the blowing-up of the projective
space at $r$ points in very general
position with exceptional divisors
$E_1,\dots,E_r$ and let $H$ be the
pullback of a hyperplane.
The {\em virtual dimension} of a 
divisor $D = dH-\sum_{i=1}^rm_iE_i$ is 
\[
 v(D) 
 := 
 \binom{d+n}{n} - \sum_{i=1}^r\binom{m_i+n-1}{n}-1
\]
and its {\em expected dimension} is 
$e(D) := \max\{v(D),-1\}$.
The divisor $D$ is {\em non-special} 
if $\dim |D| = e(D)$ or equivalently 
if $h^0(X^n_r,\mathcal O(D))\cdot
h^1(X^n_r,\mathcal O(D))=0$.
When $n=2$ there are three well known 
conjectures about divisors:
\begin{enumerate}
\item
an effective divisor 
$D$ which has intersection product at 
least $-1$ with any $(-1)$-curve of 
$X_r^2$ is {\em non-special};
\item
the only negative curves of $X_r^2$ are 
$(-1)$-curves;
\item
the divisor class $\sqrt{r}H-\sum_{i=1}^rE_i$ 
is nef.
\end{enumerate}
The first one is known as
the Segre–Harbourne–Gimigliano–Hirschowitz
(SHGH) Conjecture, while the latter is the 
famous Nagata Conjecture.
It is known that $(1)\Rightarrow 
(2)\Rightarrow (3)$ (see~\cite{chmr})
and that the first conjecture admits other  
equivalent formulations (see~\cite{cm}).
%A nef and big divisor $D$ of $X_r^2$ is 
%{\em asymptotically special} if $nD$ is 
%special for all $n\gg 0$ (see 
%Section~\ref{sec:as}).
Our first result is the following.
\begin{introthm}
\label{asymp}
Let $r$ be a positive integer. Then the 
following are equivalent.
\begin{enumerate}
\item
The only negative curves of $X_r^2$ are 
$(-1)$-curves.
\item
$X_r^2$ does not contain nef and big 
divisors which are asymptotically 
special.
\end{enumerate}
\end{introthm}

In order to prove the theorem we
show that on an algebraic surface 
curves which are orthogonal
to not asymptotically special nef and big 
divisor have arithmetic genus at most one.
This result can be used to bound the Mori
cone of the surface in some cases.
As a consequence of this result we can prove
the following equivalence.

\begin{introthm}
\label{nef}
Let $r$ be a positive integer. Then the 
following are equivalent.
\begin{enumerate}
\item
The SHGH conjecture holds on $X^2_r$.
\item
Each nef class of $X_r^2$ is non-special.
\end{enumerate}
\end{introthm}

The knowledge of the nef cone of $X_r^2$
is thus of a fundamental importance
for all the above conjectures.
Recent results about its shape are given 
in~\cite{df,chmr,cmr2}. We proved similar 
results to those in~\cite{cmr2} in an
effort to understand the nef cone in the
special case of $X_{10}^2$, see 
Remark~\ref{remeff}.
Finally we discuss the higher dimensional 
case proving the following.

\begin{introthm}
\label{thm:Pn}
If $r < 2^n$ then any nef divisor of $X_r^n$
is asymptotically non-special.
\end{introthm}

The paper is organized as follows.
Section one is devoted to the definition
and first properties of asymptotically
special nef and big divisors on surfaces.
In Section~\ref{sec:pf12} we prove Theorem~\ref{asymp}
and Theorem~\ref{nef}, while in Section~\ref{sec:pf3}
we prove Theorem~\ref{thm:Pn}.

\section{Asymptotically special divisors}
\label{sec:as}

Let $X$ be a normal $\mathbb Q$-factorial
projective variety and let $D$ be a divisor 
of $X$. We say that $D$ is {\em asymptotically 
special} if 
\[
 h^1(X,\mathcal O_X(nD))>0,
 \quad
 \text{for $n\gg 0$}.
\]
Similarly we say that $D$ is 
{\em asymptotically non-special} 
if $h^1(X,\mathcal O_X(nD))=0$
for $n\gg 0$.
If $X$ is a surface and 
$D^2>0$ then the quadratic
function $D^\perp\to\mathbb Q$, defined by
$E\mapsto p_a(E) := \frac{1}{2}(E^2+E\cdot K_X) + 1$ is bounded from above.
Indeed, by the Hodge index theorem,  
the intersection form on $D^\perp$
is negative definite, so that the  
function $p_a$ is concave. Thus
the following number is well defined
\[
 p_a(D^\perp)
 :=
 \begin{cases}
 0 & \text{if $D$ is ample}\\
 \max\{p_a(E)\, :\, 
 E\neq 0\text{ effective and }E\cdot D = 0\}
 & \text{otherwise}.
 \end{cases}
\]
The main result of this section is the following.

\begin{proposition}
\label{pa}
Let $X$ be a smooth projective surface and 
let $D$ be a nef and big Cartier divisor of $X$. 
\begin{itemize}
\item
If $p_a(D^\perp) = 0$, then $D$ is 
asymptotically non-special.
\item
If $p_a(D^\perp) = 1$ and $D$ is
disjoint from all curves in $D^\perp$,
or $p_a(D^\perp) \geq 2$,
then $D$ is asymptotically special.
\end{itemize}
\end{proposition}

We postpone the proof of the proposition
until the end of this section and meanwhile
we discuss some of its consequences.

\begin{example}
\label{mix}
In case $p_a(D^\perp) = 1$ the divisor
$D$ can be neither asymptotically
special nor asymptotically non-special.
%A nef and big divisor $D$ can admit 
%infinitely many positive values of $n$
%for which $h^1(X,\mathcal O_X(nD))$ vanishes
%as well as infinitely many positive values of 
%$n$ for which the $h^1$ does not vanish.
As an example, consider
the blow-up $X$ of $\mathbb P^2$ at $10$
points $p_1\dots,p_{10}$ on a smooth plane cubic. 
Denote by $H$ the pull back 
of a line, by $E_i$ the exceptional divisor over $p_i$ and let $C$ be 
the strict transform of the given 
cubic (i.e. $C$ is the unique 
element in $|-K_X|$).
We consider a divisor 
$D := 10H-\sum_{i=1}^{10}3E_i$
which restricts to a $2$-torsion 
class on $C$ and we claim that
\[
 h^1(X,\mathcal O_X(nD)) 
 = 
 \begin{cases}
  0 & \text{if $n$ is odd}\\
  1 & \text{if $n$ is even.}
\end{cases}
\]
First of all observe that 
by~\cite[Thm.~III.1]{har},
$C$ is contained in the base locus
of $nD$ if and only if the restriction 
of $nD$ to $C$ is non-trivial, i.e. if
and only if $n$ is odd. If this is the
case, $nD-C$ is effective and
satisfies $(nD-C)\cdot(-K_X) 
= -C^2 = 1$, so 
that $h^1(X,\mathcal O_X(nD)) = 0$
by~\cite[Thm.~III.1]{har}.
On the other hand, if $n$ is even, 
the restriction 
of $nD$ to $C$ is trivial, and this 
implies that $C$ is not contained in the 
base locus of $nD$. Therefore 
$nD$ has no fixed part and 
by~\cite[Thm.~III.1]{har} we
conclude that
$h^1(X,\mathcal O_X(nD)) = 1$. 
\end{example}

\begin{corollary}
\label{cor:spe}
Let $X$ be a normal $\mathbb Q$-factorial
projective surface with rational singularities
and let $C\subseteq X$ be an irreducible 
and reduced curve with $C^2<0$ and
$p_a(C) > 1$. Then there exists an asymptotically 
special nef and big divisor $D$ of $X$ such that 
$D\cdot C = 0$.
\end{corollary}
\begin{proof}
Being $C^2 < 0$, the intersection $C^\perp\cap {\rm Nef}(X)$
is a facet of the nef cone of $X$. Let $D$ 
be a Cartier divisor whose class lies in the relative interior 
of this facet. Then $D$ is nef and big and
$C\cdot D = 0$.
Since $p_a(D^\perp)\geq p_a(C) > 1$, Proposition~\ref{pa} 
implies that $D$ is asymptotically special.
\end{proof}

In order to prove Proposition~\ref{pa}
we need some preliminary lemmas. 
We begin with the following result which is well known in literature.

\begin{lemma}
 \label{ort}
Let $X$ be a normal $\mathbb Q$-factorial
projective surface and let $C_1,\dots,C_n
\subseteq X$
be distinct irreducible and reduced curves
whose intersection matrix is negative
definite. Then the classes $[C_1],\dots,
[C_n]$ are linearly independent in 
$N^1(X)$. In particular if $D$ is 
a nef and big divisor of $X$, then 
$D^\perp$ contains finitely many classes of 
irreducible curves. 
\end{lemma}
\begin{proof}
Let $D := \sum_ia_iC_i \sim 0$.
Write $D = A - B$, where $A$
is the sum over the positive 
coefficients of $D$ and $B$ over 
the negative ones. 
Then $0 = D^2 = A^2-2A\cdot B
+B^2$, and the fact that the three 
summands
are all non-positive, imply that both 
$A$ and $B$ are linearly equivalent 
to $0$. Since both $A$ and $B$ are
effective divisors and $X$ is complete, 
it follows that $A = B = 0$, so that 
$a_i=0$ for any $i$.

To prove the second statement, observe that
by the Hodge index theorem and
the fact that $D^2>0$, the intersection
form on $D^\bot$ is negative definite,
so that the statement follows.
\end{proof}

A version of the following lemma is proved
in~\cite[Thm. 7.2.1]{Is} for smooth
surfaces. The proof works verbatim in the
$\mathbb Q$-factorial case and we include it
here for the sake of completeness of argument.

\begin{lemma}
\label{lem:block}
Let $C_1,\dots,C_n$ be irreducible curves
on a normal projective $\mathbb Q$-factorial 
surface whose intersection matrix is 
negative definite. Then there exists an
effective divisor $E := \sum_{i=1}^na_iC_i$
such that $E\cdot C_i < 0$ for each $i$.
\end{lemma}
\begin{proof}
Since the intersection matrix is non-singular
there exists a divisor $E = \sum_{i=1}^na_iC_i$
such that $E\cdot C_i<0$ for any $i$.
Let $E = E_1 - E_2$, where $E_1, E_2$ are effective
divisors with no common support. Then
$0 \geq E\cdot E_2 = (E_1-E_2)\cdot E_2
= E_1\cdot E_2-E_2^2\geq 0$ implies $E_2=0$,
so that $E$ is effective.
\end{proof}

\begin{lemma}
\label{equiv}
Let $X$ be a normal $\mathbb Q$-factorial
projective surface with rational singularities
and let $D$ be a nef and big Cartier divisor 
of $X$ which is not ample.
Then there exists an effective divisor $E$
in $D^\perp$ such that 
\[
 H^1(X,\mathcal O_X(nD))\simeq 
 H^1(E,\mathcal O_E(nD)).
\]
holds for $n\gg 0$.
\end{lemma}
\begin{proof}
By Lemma~\ref{ort} there are only
finitely many irreducible and reduced
curves $C_1,\dots,C_s$ in $D^\bot$.
By Lemma~\ref{lem:block} there exists 
an effective divisor $E$, supported at these
curves, such that $E\cdot C_i < 0$ for any $i$.
Up to replace $E$ with a positive 
multiple, we can assume $E$ to be Cartier
and $(-E-K_X)\cdot C_i>0$ for any $i$.
For any $n>0$, we have the following 
exact sequence of sheaves 
\[
 \xymatrix{
 0\ar[r] &
 \mathcal O_X(nD-E)\ar[r] &
 \mathcal O_X(nD)\ar[r] &
 \mathcal O_E(nD)\ar[r] &
 0.
 }
\]
%Since $D$ is semiample, there is 
%an isomorphism $\mathcal O_\Gamma(nD)
%\simeq \mathcal O_\Gamma$ for 
%$n\gg 0$.
%Up to replacing $\Gamma$ with a suitable
%positive multiple we have that the divisor
By the above assumption on $E$ the
divisor $nD-E-K_X$ is nef and big for $n\gg 0$.
Thus we conclude by passing to cohomology 
and using the Kawamata-Viehweg vanishing 
theorem for Cartier divisors on varieties 
with rational singularities~\cite{ko}.
\end{proof}

\begin{proof}[Proof of Proposition~\ref{pa}]
If $D$ is ample, then the statement follows 
from the Serre vanishing theorem
\cite[Thm. 1.2.6]{laz}.
Assume now that $D$ is not ample.
If $p_a(D^\perp) = 0$ then 
$h^1(E,\mathcal O_E) = 0$ 
for any effective divisor $E$
in $D^\perp$, by~\cite[Lem. 3.3]{ba}. 
Thus
\[
 h^1(E,\mathcal O_E(nD)) = 0
\]
for any $n>0$ by~\cite[Cor. 3.6]{ba}.
We conclude that $D$ is 
asymptotically non-special
by Lemma~\ref{equiv}.

Assume now $p_a(D^\perp)\geq 1$.
Let $E'$ be an effective divisor in
$D^\perp$ such that $p_a(E') = p_a(D^\perp)$.
Since $\deg \mathcal O_{E'}(nD) = 0$
it follows that
\[
 \chi(E',\mathcal O_{E'}(nD)) 
 = 
 \chi(E',\mathcal O_{E'})
 =
 1-p_a(E')\leq 0,
\]
where the first equality is by
\cite[\href{https://stacks.math.columbia.edu/tag/0BRZ}{Tag 0BRZ}]{stacks-project} 
and the second is by the definition
of arithmetic genus.
If $p_a(E')\geq 2$ or $p_a(E')=1$ and
$\mathcal O_{E'}(nD)\simeq \mathcal O_{E'}$
we deduce $h^1(E',\mathcal O_{E'}(nD)) > 0$.
The divisor $E$ in Lemma~\ref{equiv}
can be chosen so that $E'\leq E$.
If we denote by $I$ be the ideal sheaf of $E'$ in $\mathcal O_{E}$, we have
an exact sequence of sheaves 
\[
 \xymatrix{
 0\ar[r] & 
 I(nD)\ar[r] &
 \mathcal O_{E}(nD)\ar[r] &
 \mathcal O_{E'}(nD)\ar[r] &
 0.
 }
\]
Passing to cohomology and using the
fact that we are in dimension one
we get the surjection 
$H^1(E,\mathcal O_{E}(nD))
\to H^1(E',\mathcal O_{E'}(nD))$, 
which implies
$h^1(E,\mathcal O_{E}(nD))>0$,
so that $D$ is asymptotically 
special by Lemma~\ref{equiv}.
\end{proof}

\section{Proof of Theorems~\ref{asymp}
and ~\ref{nef}}
\label{sec:pf12}
Recall that $\pi\colon X_r^2\to\mathbb P^2$
is the blowing-up of the projective
plane at $r$ points in very general
position with exceptional divisors
$E_1,\dots,E_r$ and $H$ is the
pullback of a hyperplane.
We denote by ${\rm Eff}(X_r^2)$ the cone of classes
of effective divisors of $X_r^2$, by 
$\overline{\rm Eff}(X_r^2)$ its closure 
in the Euclidean topology and by 
${\rm Nef}(X_r^2)$ the dual of the latter
closed cone with respect to the bilinear
form defined by the intersection product.
Elements of $\overline{\rm Eff}(X_r^2)$ are
called pseudoeffective classes, while
elements of ${\rm Nef}(X_r^2)$ are nef classes.
Let
\[
 Q(X_r^2) :=
 \{D\in {\rm Pic}(X_r^2)_{\mathbb R}\, :\,
 D^2 \geq 0\text{ and }D\cdot H\geq 0\}
\]
be the {\em positive light cone} of $X_r^2$.
Recall the following inclusions
\[
 {\rm Nef}(X_r^2)
 \subseteq
 Q(X_r^2)
 \subseteq
 \overline{\rm Eff}(X_r^2).
\]
Let $W(X_r^2)$ be the subgroup of isometries
of ${\rm Pic}(X_r^2)$ generated  by the reflections
\[
 D\mapsto D+(D\cdot R)R,
\]
where $R$ is one of the {\em fundamental roots}
(classes of self-intersection $-2$):
$E_1-E_2,\dots,E_{r-1}-E_r,H-E_1-E_2-E_3$.
It is not difficult to show that any such
reflection preserves all the cones defined above.
The choice of a distinguished representative 
for an orbit of $W(X_r^2)$ contained in the 
effective cone leads to the following definition.

\begin{definition}
A class $D := dH-\sum_{i=1}^rm_iE_i
\in {\rm Pic}(X_r^2)_{\mathbb R}$ is {\em pseudostandard} if
\[
 d\geq m_1+m_2+m_3
 \quad
 \text{and}
 \quad
 m_1\geq\cdots\geq m_r.
\]
If in addition $m_r\geq 0$ then the
class $D$ is {\em standard}.
%if one of the following occur:
%\begin{enumerate}
%\item
%$D$ is a positive multiple of %$E_1$;
%\item
%$d\geq m_1+m_2+m_3$ and 
%$m_1\geq\cdots\geq m_r\geq 0$.
%\end{enumerate}
\end{definition}

The first inequality in the definition is equivalent
to $D\cdot (H-E_1-E_2-E_3)\geq 0$, while the 
ordering of the multiplicities comes from 
$D\cdot (E_i-E_{i+1})\geq 0$.
%Observe that $W(X_r^2)$ contains all the
%permutations of the $r$ exceptional 
%divisors. When the first three points are 
%the fundamental ones of the projective 
%plane, then the reflection by the last
%root is the map induced by the  
%quadratic Cremona transformation 
%$[x_0:x_1:x_2] \mapsto 
%[x_0^{-1}:x_1^{-1}:x_2^{-1}]$.
%Recall that $W(X_r^2)$ preserves the nef, light
%and effective cone.
In what follows, with abuse of notation, we will 
denote by $-K_s$ the class $3H-\sum_{i=1}^sE_i$,
for any $s\leq r$.

%The next proposition characterizes standard classes
%within those with non-positive self-intersection
%and non-positive intersection with $K_r$.

\begin{lemma}
\label{std}
Let $D\in {\rm Pic}(X_r^2)_{\mathbb R}$ be a 
standard class.
If $D^2\leq 0$ and $D\cdot K\leq 0$ then
$D$ is a positive multiple of either 
$H-E_1$ or $-K_9$. 
%In particular any standard 
%class in $Q_0(X_r^2)$ is nef.
\end{lemma}
\begin{proof}
Let $D = dH - \sum_{i=1}^rm_iE_i$.
Since $D$ is standard we have 
that $d\geq m_1+m_2+m_3$ and 
 $m_1\geq\cdots\geq m_r\geq 0$.
The inequalities $D^2 \leq 0$ and 
$D\cdot K_r \leq 0$ imply 
$D^2+m_3D\cdot K_r \leq 0$, that is
\begin{align*}
 &d(d-3m_3)-\sum_{i=1}^r(m_i^2-m_3m_i)=0\\
 \Rightarrow 
 &d(d-3m_3) - (m_1^2-m_1m_3)-(m_2^2-m_2m_3)
 = \sum_{i=3}^r(m_i^2-m_3m_i)\leq 0.
\end{align*}
Since $d\geq m_1+m_2+m_3$ the left hand 
side is greater than or equal to
$(m_1+m_2+m_3)(m_1+m_2+m_3-3m_3)-
 (m_1^2-m_1m_3)-(m_2^2-m_2m_3)
 = 2(m_1m_2-m_3^2)$, from which we 
 deduce
\[
 0\leq 
 2(m_1m_2-m_3^2)\leq 
 \sum_{i=3}^r(m_i^2-m_3m_i)\leq 0.
\]
Therefore either $m_2=\cdots=m_r=0$
or all the multiplicities are equal. 
In the first case $D^2\leq 0$ imply
$d^2\leq m_1^2$, so that $d=m_1$ and
$D$ is a multiple of $H-E_1$. 
In the second case
$D = dH - \sum_{i=1}^rmE_i$.
The inequalities $0\leq d^2-3md
= D^2+mD\cdot K_r\leq 0$ imply $d=3m$
and $D^2 = D\cdot K_r = 0$, which proves
the statement.
\end{proof}

\begin{corollary}
\label{cor:neg}
The surface $X_r^2$ does not contain $(-2)$-curves
and the classes of $(-1)$-curves form an orbit
for $W(X_r^2)$.
\end{corollary}
\begin{proof}
Let $C$ be a $(-1)$ or $(-2)$-curve. 
Since the action of $W(X_r^2)$ preserves
the effective cone, we can assume the
class $[C] := dH-\sum_{i=1}^rm_iE_i$ to be 
effective and pseudostandard.
By Lemma~\ref{std} this class cannot be
standard. This immediately implies
$d=0$, so that by the irreducibility of $C$ one 
deduces $m_1=\cdots m_{r-1}=0$ and $m_r=-1$,
which proves the statement.
\end{proof}

\begin{lemma}
\label{-1lem}
If $C$ is an irreducible and reduced 
curve of $X_r^2$ with $C^2<0$ and 
$p_a(C)\leq 1$, then $C$ is a 
$(-1)$-curve.
\end{lemma}
\begin{proof}
Let $\mathcal X\to \Delta$
be a flat family whose general fiber
is of type $X_r^2$ and whose special
fiber $Y$ over $0\in\Delta$ is an 
anticanonical rational surface with
a smooth irreducible member in
$|-K_Y|$. The existence of such 
a degeneration is guaranteed by sending
the $r$ points on a smooth plane cubic.
The curve $C$ degenerates to an effective
divisor $\Gamma$ of $Y$. We consider two
possibilities.

If $p_a(C) = 0$
then any irreducible curve in the support 
of $\Gamma$ is smooth rational, so that 
it must have non-negative intersection 
product with $-K_Y$, because $|-K_Y|$
contains a smooth genus $1$ curve.
Thus we deduce $-K_r\cdot C = 
-K_Y\cdot \Gamma \geq 0$.
By the genus formula it follows that 
$C$ is either a $(-1)$-curve or 
a $(-2)$-curve. The second case is 
ruled out by Corollary~\ref{cor:neg}.

Assume now $p_a(C) = 1$. Reasoning 
as in the previous case one concludes
that $C$ must be smooth of genus $1$.
Restricting the degeneration 
$\pi\colon\mathcal X\to\Delta$
to the degeneration of $C$ to $\Gamma$
one obtains an elliptic surface 
$\mathcal S\to\Delta$, with 
special fiber $\Gamma$. Up to
shrinking $\Delta$ we can assume the
singularities of $\mathcal S$ to be
only in the central fiber. Let 
\[
 \xymatrix@C=15pt@R=15pt{
  \tilde{\mathcal S}\ar[rr]^-\eta\ar[rd]_-{\tilde\pi} && 
  \mathcal S\ar[ld]^-\pi\\
  & \Delta &
 }
\]
be a minimal resolution of singularities.
Denote by $\tilde\Gamma$ the 
special fiber of $\tilde\pi$.
Since $\tilde\Gamma$ contains an irreducible 
curve $E$ of genus $1$, by~\cite[\S V.7]{bhpv}
we deduce $\tilde\Gamma = nE$ 
for some $n>0$. In this case we have
$\tilde{\mathcal S} = \mathcal S$ since
all the exceptional divisors of 
$\eta$ should appear in the special fiber.
Thus $\Gamma = nE$.
On the other hand $-K_Y\cdot \Gamma 
= -K_r\cdot C = C^2 < 0$ implies 
$-K_Y\cdot E < 0$, so that $-K_Y\sim E$,
because $E$ is irreducible and $|-K_Y|$
contains an irreducible curve.
Thus $-nK_r\sim C$, a contradiction
since either $r\leq 9$ and $(-K_r)^2\geq 0$ or $r\geq 10$ and 
the class of $-K_r$ lies outside
the pseudoeffective cone.
\end{proof}

\begin{proof}[Proof of Theorem~\ref{asymp}]
We prove $(1)\Rightarrow (2)$.
Let $D$ be a nef and big divisor of 
$X_r^2$. By hypothesis the only negative
curves in $D^\perp$ are $(-1)$-curves.
By the Hodge index theorem and the fact that 
$D^2>0$, the intersection matrix of these
curves is negative definite. In particular
these curves are disjoint, since $(E+E')^2
\geq 0$ if $E,E'$ are $(-1)$-curves 
with $E\cdot E' > 0$.
Contracting these curves we get a birational
morphism $X_r\to X_s$, whith $s\leq r$, and 
$D$ is the pullback of an ample divisor of $X_s$,
so that it is not asymptotically special.

We prove $(2)\Rightarrow (1)$.
Assume that there is a negative curve
$C$ on $X_r^2$ which is not a $(-1)$-curve.
By Lemma~\ref{-1lem} we have $p_a(C)>1$.
Thus by Corollary~\ref{cor:spe} the
surface $X_r^2$ would contain an asymptotically 
special divisor, a contradiction.
\end{proof}

Before entering the proof of 
Theorem~\ref{nef} we recall that
the virtual dimension of a divisor 
$D = dH-\sum_{i=1}^nm_iE_i$ is 
\[
 v(D) 
 = \binom{d+2}{2} - \sum_{i=1}^r\binom{m_i+1}{2}-1
 = \frac{1}{2}(D^2-D\cdot K_r).
\]

\begin{proof}[Proof of Theorem~\ref{nef}]
The implication $(1)\Rightarrow (2)$ is 
obvious.

We prove $(2)\Rightarrow (1)$.
Since every nef class is non-special,
it is also asymptotically non-special,
so that, by Theorem~\ref{asymp}, 
the only negative curves of
$X_r^2$ are $(-1)$-curves.
Let $D$ be a divisor such that 
$D\cdot E\geq -1$ for any $(-1)$-curve
$E$. Then $D \sim M + F$, where $F$
is the sum of all $(-1)$-curves 
which have intersection product $-1$
with $D$. Observe that if $E$, $E'$
are two distinct $(-1)$-curves in $F$
then both curves are in the base locus
of $|D|$ so that $E\cdot E'=0$.
It follows that $F$ is a reduced divisor
(each curve in its support appears with
coefficient $1$) and $M\cdot F = 0$.
Thus
\begin{align*}
 v(D) 
 &= \frac{1}{2}(D^2-D\cdot K_r)\\[5pt]
 &= \frac{1}{2}((M+F)^2-(M+F)\cdot K_r)\\[5pt]
 &= v(M) + v(F) + M\cdot F\\[5pt]
 &= v(M).
\end{align*}
Observe that $M$ is
nef because the only negative curves 
of $X_r^2$ are $(-1)$-curves. Thus 
$M$ is non-special and we deduce 
$\dim |D| = \dim |M| = v(M) = v(D)$,
which shows that $D$ is non-special
as well.
\end{proof}

\begin{remark}
\label{remeff}
As a consequence of Lemma~\ref{std} one can show
that any divisor class of 
${\rm Pic}(X_{10}^2)_{\mathbb R}$
satisfying the equalities 
$D^2=D\cdot K_{10} = 0$ is nef
(see~\cite{df} and~\cite{cmr2}).
The argument goes like this.
First of all observe that for any
$r\geq 10$
a class $D\in{\rm Pic}(X_r^2)$ satisfying 
$D^2=D\cdot K_r = 0$ is effective. Indeed,
by Riemann-Roch either $D$ or $K-D$ is effective
but $(K_r-D)\cdot H<0$.
If $D$ is in pseudostandard 
form $D = D_0 + m_1E_1+\dots+m_sE_s$, 
with $D_0$ effective, $D_0\cdot E_i=0$ 
and $m_i>0$ for any $i$.
If we denote by $\pi\colon X_r^2\to X_{r-s}^2$
the contraction of $E_1,\dots,E_s$, then
$D_0 = \pi^*(B)$ for an effective 
divisor class $B\in{\rm Pic}(X_{r-s}^2)$. 
Thus
\[
 0 = D\cdot K_r 
 = \left(D_0+\sum_{i=1}^sm_iE_i\right)\cdot K_r
% = \pi^*(B)\cdot \pi^*(K_{r-s})-\sum_{i=1}^sm_i
 = B\cdot K_{r-s}-\sum_{i=1}^sm_i
\]
implies that $B\cdot K_{r-s}=\sum_{i=1}^sm_i>0$.
Thus $-K_{r-s}$ is not nef, so that $r-s\geq 10$.
We deduce that if $r=10$ then $s$ must be zero
or in other words $D$ is a standard divisor class.
By Lemma~\ref{std} we conclude that $D$ is nef.
To conclude one observes that such classes define 
a quadric in ${\rm Pic}(X_{10}^2)_{\mathbb R}$
whose rational points are dense.
\end{remark}

\begin{example}
Let $B := dH-\sum_{i=1}^8m_iE_i$ be such that 
$b := \frac{1}{4}(2B^2-(B\cdot K_8)^2) > 0$ and 
let $a := -\frac{1}{2}B\cdot K_8$. 
Then the following class
\[
 D := B-(a+\sqrt b)E_9-(a-\sqrt b)E_{10}
\]
satisfies $D^2 = D\cdot K_{10} = 0$ so that it is nef.
\end{example}

\section{Higher dimension}
\label{sec:pf3}
In this section we consider the asymptotical 
behavior of semiample divisors in $X_r^n$
with the purpose of studying possible 
generalizations of the conjectures described 
in the Introduction.
A possible generalization of these
conjectures could be:
\begin{conjecture}
\label{con:n}
The blow-up $X^n_r$ of $\mathbb P^n$ at $r$ points
in very general position does not contain 
nef divisors which are asymptotically special. 
\end{conjecture}
We are going to prove 
that Conjecture~\ref{con:n}
holds if $r<2^n$ (Theorem~\ref{thm:Pn}),
but first we need some preliminary results.

%However it is unclear to us if this 
%statement would be equivalent to one
%about the structure of the Mori cone
%of these blowups. 
%On the other hand it is not true that 
%any nef and big class on $X_r^n$ is 
%non-special, when $n>2$, as shown in 
%Example~\ref{14pts}.

\begin{definition}
Given a semiample Cartier divisor $D$ on a 
normal variety $X$ let
\[
 m_D := \gcd(n\in\mathbb N\, :\, 
 nD\text{ is base point free}).
\]
\end{definition}
For any positive integer $m$ denote by 
$\phi_m\colon X\to Y_m$
the rational map defined by the linear system 
$|mD|$. By~\cite[Thm. 2.1.27]{laz} there is 
a morphism with connected fibers 
\[
 f_D\colon X\to Y
\]
onto a normal variety $Y$, such that 
$\phi_m = f_D$ and $Y_m = Y$ for any 
sufficiently big multiple $m$ of $m_D$.
Moreover $m_DD = f_D^*A$ for some 
ample Cartier divisor $A$ on $Y$.
\begin{remark}
The semiample and big divisor $D$ of 
Example~\ref{mix} has $m_D = 2$. Indeed
any odd multiple of $D$ restricts to a 
non-trivial degree zero class on the 
elliptic curve $C\in |-K_X|$, while 
the even multiples restrict trivially to $C$.
The latter fact implies that $C$ is not
in the base locus of $2D$ and by
\cite[Thm. III.1]{har} we conclude that
$2D$ is base point free.    
\end{remark}

\begin{proposition}
\label{pro:cri}
Let $X$ be a normal projective variety
and let $D$ be a semiample big Cartier 
divisor of $X$ with $m_D=1$.
If $R^1{f_D}_*\mathcal O_X$ is trivial
then $D$ is asymptotically non-special
and otherwise it is asymptotically special. 
\end{proposition}
\begin{proof}
Let $f := f_D$ and  
recall that $D = f^*A$, with $A$ ample 
on $Y$. By Serre's vanishing
\cite[Thm. 1.2.6]{laz} the higher
cohomology groups of $\mathcal O_Y(mA)$ 
vanish for all $m\gg 0$.
Since the Grothendieck-Leray spectral 
sequence
\[
 E_2^{p,q}(D) 
 := 
 H^p(Y,R^qf_*\mathcal O_X(D))
 \Rightarrow 
 H^{p+q}(X, \mathcal O_X(D))
\]
is a first 
quadrant one, by~\cite[Prop. 3.6.2]{Is} we 
have the following five terms exact sequence:
\[
 \xymatrix@R=2pt@C=10pt{
  0 \ar[r] & 
  E^{1,0}_2(mD) \ar[r] & 
  H^1(\mathcal O_X(mD)) \ar[r] &
  E^{0,1}_2(mD) \ar[r] & 
  E^{2,0}_2(mD) \ar[r] & 
  H^2(\mathcal O_X(mD)).
 }
\]
By the projection formula we have
$f_*\mathcal O_X(mD) = 
f_*\mathcal O_X\otimes \mathcal O_Y(mA) 
\simeq \mathcal O_Y(mA)$, so that the first 
and fourth cohomology groups 
$E^{i,0}_2(mD) = H^i(Y,f_*\mathcal O_X(mD))$
vanish for all $m\gg 0$. Thus the second
and third cohomology groups are isomorphic.
Again by the projection formula we have 
$R^1f_*\mathcal O_X(mD) \simeq 
R^1f_*\mathcal O_X\otimes \mathcal O_Y(mA)$
so that
\[
 H^1(X,\mathcal O_X(mD))
 \simeq
 H^0(Y,R^1f_*\mathcal O_X\otimes \mathcal O_Y(mA))
 \qquad
 \text{for all $m\gg 0$}.
\]
If $R^1f_*\mathcal O_X$ is trivial 
we conclude that $D$ is asymptotically 
non-special.
On the other hand if
$R^1f_*\mathcal O_X$ is non-trivial,
since it is a coherent sheaf and 
$A$ is ample, by Serre's
Theorem~\cite[Thm. 1.2.6]{laz}
the sheaf 
$R^1f_*\mathcal O_X\otimes \mathcal O_Y(mA)$
is generated by global sections for all 
$m\gg 0$. We conclude that $D$ is 
asymptotically special.
\end{proof}
Before proving Theorem~\ref{thm:Pn} we are 
now going to discuss an example 
of a nef and big divisor 
which is special but asymptotically 
non-special, on the blow up of 
$\mathbb P^4$ at $14$ points.

\begin{example}
\label{14pts}
Let $D$ be the divisor $2H-\sum_{i=1}^{14}E_i$ on $X^4_{14}$. 
The virtual dimension of $mD$, for 
$m = 1,2,3$ is $0,-1,-1$ respectively, so that
$2D$ and $3D$ are special. Moreover
a computer calculation shows that 
$4D$ is non-special of dimension $4$ and
its base locus is zero-dimensional. Thus 
$D$ is semiample by~\cite[Rem. 2.1.32]{laz} and it is big because $D^4 = 2$.
We now show that $mD$ is non-special 
for $m\geq 5$ using a degeneration 
argument.
Let $\mathcal X\to\mathbb A^1$ be a 
flat family whose general fiber is 
$X^4_{14}$ and whose special fiber $Y$
is the blowing-up of $\mathbb P^4$
at $14$ general points on the complete
intersection of three general quadrics.
Denote by $Y_1\subseteq Y_2\subseteq Y_3\subseteq Y_4 = Y$ the strict transforms
of the complete intersections of three,
two and one quadrics respectively, 
plus the central fiber
of the degeneration. By abuse of notation
we denote by $D$ the specialized divisor on $Y$. 
Observe that for any 
$2\leq i\leq 4$,
$Y_{i-1}\in |D|_{Y_i}|$,
so that we have an exact sequence
\[
 \xymatrix{
  0\ar[r] &
  \mathcal O_{Y_i}((m-1)D)\ar[r] &
  \mathcal O_{Y_i}(mD)\ar[r] &
  \mathcal O_{Y_{i-1}}(mD)\ar[r] &
  0.
 }
\]
Since $Y_1$ is a smooth curve of genus $5$ 
and $\deg(D|_{Y_1}) = D\cdot Y_1 = 2$ it follows that 
$mD|_{Y_1}$ is non-special for $m\geq 5$.
From the above exact sequence we deduce 
that $h^1(\mathcal O_{Y_2}(mD))
\leq h^1(\mathcal O_{Y_2}((m-1)D))
\cdots\leq h^1(\mathcal O_{Y_2}(5D))$.
A computer calculation shows that
$5D|_{Y_i}$ is non-special for $2\leq i
\leq 4$. 
A repeated use of the above
exact sequence shows 
that for any $m\geq 5$, 
$mD$ is non-special on 
$Y_4$ and, by 
semicontinuity of cohomology, it is 
non-special on $X_{14}^4$ 
too.
%The case $m=4$ can be proved again 
%by computer calculation using a set of 
%points in general position.
\end{example}

\begin{proof}[Proof of Theorem~\ref{thm:Pn}]
Let $H,E_1,\dots,E_r$ be a basis of 
${\rm Pic}(X_r^n)$ consisting of the
pullback of a hyperplane together with 
the exceptional divisors. If we denote by
$h,e_1,\dots,e_r$ the dual basis 
in ${\rm N}_1(X_r^n)$, we have that 
the Mori cone of $X_r^n$ is polyhedral,
generated by the classes $h-e_i-e_j$ and $e_k$, 
for $i,j,k\in \{1,\dots,r\}$ and $i < j$.
(see~\cite{CLO}).
The proof of this fact goes as follows:
the cone generated by these classes
is dual to the cone generated by the 
classes
\[
 H, 
 \quad
 H-E_i,
 \quad
 2H-\sum_{i\in I}E_i,
 \]
for any subset $I\subseteq\{1,\dots,r\}$.
These classes are all semiample (and thus
nef) on the variety $Y_r^n$ obtained by 
degenerating the $r$ points to a complete
intersection of $n$ general smooth quadrics.
Then one concludes by observing that 
the nef cone can only become smaller
by degeneration.

By semicontinuity of cohomology it is now enough
to show that any nef class of $Y_r^n$ is
asymptotically non-special.
Observe that any $H-E_i$ is asymptotically 
non-special
because it is pullback of a class of $X_1^n$
which is asymptotically non-special 
by Kawamata-Viehweg vanishing.
All the nef classes of $Y_r^n$,
which are not multiples of $H-E_i$,
are big. Thus by Proposition~\ref{pro:cri}
it suffices to show that for any nef and big 
divisor class $D$ of $Y_r^n$ the sheaf
$R^1{f_D}_*\mathcal O_{Y_r^n}$ is trivial.
This is consequence of the fact
that $f_D$ can contract only exceptional 
divisors and strict transforms of lines 
through two points. In the first case
one gets a smooth point while in the second
case the image of $f_D$ has a rational
singularity because the morphism is 
locally toric and one can 
apply~\cite[Thm. 11.4.2]{toric}.
\end{proof}

\bibliographystyle{plain}

\end{document}